\documentclass[12pt,reqno]{amsart}
\usepackage{amsfonts}
\usepackage{amsthm}
\usepackage{amsmath}
\usepackage{amscd}
\usepackage[latin2,utf8]{inputenc}
\usepackage{t1enc}
\usepackage[mathscr]{eucal}
\usepackage{indentfirst}
\usepackage{graphicx}
\usepackage{graphics}
\usepackage{pict2e}
\usepackage{epic}
\numberwithin{equation}{section}
\usepackage{enumitem}
\usepackage[margin=4cm]{geometry}
\usepackage{epstopdf}
\usepackage{color}
\usepackage{dsfont}
\usepackage{pgfplots}
\usepackage{tikz}
\usepackage{float}
\usepackage{url,hyperref}

\theoremstyle{plain}
\newtheorem{Th}{Theorem}[section]

\newtheorem{Prop}[Th]{Proposition}

\theoremstyle{definition}

\newtheorem{Rem}[Th]{Remark}
\newtheorem{?}[Th]{Problem}

\providecommand{\num}[1]{\nu\left(#1\right)}

\providecommand{\norm}[1]{\left\lVert#1\right\rVert}

\providecommand{\parenth}[1]{\left(#1\right)}
\providecommand{\conv}[1]{\mathop{\rm conv}\{#1\}}

\newcommand{\R}{\mathbb{R}} 
\newcommand{\Pro}{\mathbb{P}} 
\newcommand{\Sp}{\mathbb{S}} 
\newcommand*{\ind}{\mbox{$\mathds{1}$}} 

\begin{document}

\title[Stochastic forms of Brunn's principle]{Stochastic forms of Brunn's principle}

\author[P. Pivovarov \& J. Rebollo Bueno]{Peter Pivovarov \& Jes\'us
  Rebollo Bueno}

\address{University of Missouri \\ Department of Mathematics
  \\ Columbia MO 65201}

\email{pivovarovp@missouri.edu \& jrc65@mail.missouri.edu}



\begin{abstract}
  A number of geometric inequalities for convex sets arising from
  Brunn's concavity principle have recently been shown to yield local
  stochastic formulations.  Comparatively, there has been much less
  progress towards stochastic forms of related functional
  inequalities.  We develop stochastic geometry of $s$-concave
  functions to establish local versions of dimensional forms of
  Brunn's principle \`{a} la Borell, Brascamp-Lieb, and Rinott.  To do
  so, we define shadow systems of $s$-concave functions and revisit
  Rinott's approach in the context of multiple integral rearrangement
  inequalities.
\end{abstract}

\maketitle

\section{Introduction}
\label{intro}

Brunn's concavity principle underpins a wealth of inequalities in
geometry and analysis. One can formulate it as follows: for any convex
body $K\subseteq \mathbb{R}^n$ and any direction $\theta$, the
$(n-1)$-volume of slices of $K$ by parallel translates of
$\theta^{\perp}$ is $1/(n-1)$-concave on its support, i.e.,
\begin{equation}
  \label{eqn:Brunn}
 A(t) = \lvert K\cap (\theta^{\perp}+t\theta)\rvert^{1/(n-1)}
\end{equation}
 is concave.  A far-reaching extension of this principle in analysis
is exemplified by a family of functional inequalities obtained by
Borell \cite{Bor:1975-Convex} and Brascamp-Lieb
\cite{bralieonextension}, with an alternate approach
put forth by Rinott \cite{rinonconvexity}.  These inequalities can be
formulated in terms of certain means as follows: for $a,b\geq 0$,
$s\geq-1/n$ and $\lambda\in (0,1)$, set
    \begin{equation*}
    \mathcal{M}_{\lambda}^s(a,b) =
    \begin{cases}
      \parenth{\lambda a^s + (1-\lambda)b^s}^{1/s} & \text{ if } ab\neq0 \\
      0 & \text{ if }  ab=0,
    \end{cases}
    \end{equation*}
where the cases $s\in\{-1/n,0,+\infty\}$ are defined as limits
    \begin{align*}
      \mathcal{M}_{\lambda}^{-1/n}(a,b) & =\min\{a,b\},
      &
      \mathcal{M}_{\lambda}^{+\infty}(a,b) & =\max\{a,b\},
      &
      \mathcal{M}_{\lambda}^0(a,b) & =a^{\lambda}b^{1-\lambda}.
    \end{align*}
    Then for measurable functions $f,g,h:\R^n\to[0,\infty)$ ,
      $0<\lambda<1$, and $s\geq-1/n$, if
    \begin{equation}    \label{BLL_assumption}
        h(\lambda x_1+(1-\lambda)x_2)
        \geq
        \mathcal{M}_{\lambda}^s(f(x_1),g(x_2))
    \end{equation}
for all $x_1,x_2\in\R^n$, one has
    \begin{equation}
    \label{BBL}
        \displaystyle\int_{\R^n} h
        \geq
        \mathcal{M}_{\lambda}^{s/(1+ns)}
        \parenth{
        \int_{\R^n} f, \int_{\R^n} g
        }.
    \end{equation}
The Pr\'{e}kopa-Leindler inequality \cite{leiconv,preonlog1,preonlog2}
corresponds to the logarithmically concave case $s=0$; for earlier
work on the real line, see Henstock and Macbeath
\cite{HenMac:1953-1dimBLL}.  These principles are now fundamental in
analysis, geometry and probability, among other fields.  For
their considerable impact, we refer the reader to \cite{garbru} and
the references therein.

The inequalities \eqref{BBL} stem from principles rooted in convexity.
Indeed, the standard approach to Brunn's principle \eqref{eqn:Brunn}
connects concavity of the map $t\mapsto A(t)$ to the convexity of $K$
through suitable symmetrizations, see e.g., \cite{AGM}.  In
\cite{braliebestconstants}, Brascamp and Lieb used symmetrization to
prove certain cases of \eqref{BBL}. Subsequently, they provided an
alternate inductive approach, based on the Brunn-Minkowski inequality
\cite{bralieonextension}.  Rinott provided an alternate proof,
starting with epigraphs of convex functions \cite{rinonconvexity}.
However, the inequalities ultimately do not require convexity as they
hold for measurable functions. So convexity (or concavity) of the
functions involved seems of no importance.  In this paper, our focus
is on what more can be said when the functions involved do possess
some concavity.

Our motivation stems from recent work for convex sets in which a
``local'' stochastic dominance accompanies an isoperimetric
principle. A concrete example, which can be derived via Brunn's
principle, is the Blaschke-Santal\'{o} \cite{San:1949} inequality. The
latter says that the volume of the polar of an origin-symmetric convex
body $K$ is maximized by a Euclidean ball $B$ under a constraint of
equal volume.  Proofs via symmetrization depend on variants of
\eqref{BBL}, e.g., Meyer-Pajor \cite{meypajblaschke} and Campi-Gronchi
\cite{campigronchionvolume}. In \cite{CFPP}, the first-named author,
together with Cordero-Erausquin, Fradelizi and Paouris, proved a
stochastic version in which the dominance applies ``locally'' to
random polytopes that naturally approximate $K$ and $B$ from within.
By repeated sampling, this recovers the Blaschke-Santal\'{o}
inequality by the law of large numbers.  This example is indicative of
recent research on isoperimetric inequalities: when an isoperimetric
principle for convex sets can be proved by symmetrization, it is
fruitful to instead carry out the symmetrization on product
probability spaces.  Multiple integral rearrangement inequalities of
Rogers \cite{Rog:1957-Rearrangement}, Brascamp-Lieb-Luttinger
\cite{BLL}, and Christ \cite{christestimates} then enter the picture
and yield stronger stochastic formulations.  This builds on principles
in stochastic geometry e.g., \cite{Bus:1953,camcolgroanote,Gro:1974};
see \cite{CFPP,paopivprob,paopivrand,RB:2020-Planar} for further
background.

While there is much work on stochastic isoperimetric inequalities for
convex sets, there are far fewer results about random functions.  In
\cite{PivRB:2019-PreLei}, we initiated work on the
Pr\'{e}kopa-Leindler inequality for random $\log$-concave functions.
Here we will show that the full family of inequalities \eqref{BBL}
actually have ``local'' stochastic strengthenings for functions $f$
that are $s$-concave, i.e. $f^s$ is concave on its support; when
$s<0$, this means that $f^s$ is convex.  To formulate our main result,
for each $s$-concave function $f$, we sample independent random
vectors $(X_1,Z_1),\ldots,(X_N,Z_N)$ distributed uniformly under the
graph of $f$ according to Lebesgue measure. We associate random
functions $[f]_N$, supported on the convex hull $\mathop{\rm
  conv}\{X_1,\ldots,X_N\}$, defined by
    \begin{subequations}
    \begin{align*}
        [f]_N(x)
        &=
        \begin{cases}
            \inf\{z^{1/s}:(x,z)\in \mathcal{P}_{f,N}\}, & \text{if $s<0$} \\[.2em]
            \sup\{e^{z}:(x,z)\in \mathcal{P}_{f,N}\}, & \text{if $s=0$}, \\[.2em]
            \sup\{z^{1/s}:(x,z)\in \mathcal{P}_{f,N}\}, & \text{if $s>0$},
        \end{cases}
\intertext{where}
        \mathcal{P}_{f,N}
        &=
        \begin{cases}
            {\rm
              conv}\{(X_1,{Z^s_1}),\ldots,(X_N,{Z^s_N})\}, & \text{if $s\neq0$}. \\[.2em]
            {\rm
              conv}\{(X_1,\log{Z_1}),\ldots,(X_N,\log{Z_N})\}, & \text{if $s=0$}. 
        \end{cases}
    \end{align*}
    \end{subequations}
    In other words, when $s=0$ or $s>0$, $[f]_N$ is the least
    $\log$-concave or $s$-concave function, respectively, satisfying
    $[f]_N(X_i)\geq Z_i$; similarly, when $s<0$, $[f]_N$ is the
    greatest $s$-concave function with $[f]_N(X_i)\leq Z_i$. See
    Figure \ref{fig:log} for the case $s=0$.

\begin{figure}
  \label{fig:log}
\includegraphics[scale=0.9]{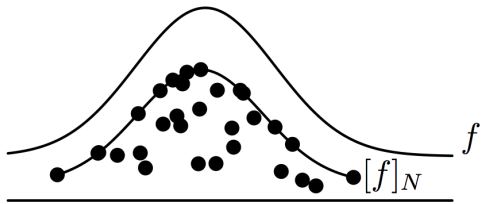}
\caption{Stochastic approximation of a $\log$-conave function $f$ by
  its least $\log$-concave majorant $[f]_N$ above a random sample
  under the graph of $f$.}
\end{figure} 

With this notation, we can state our main result, which we formulate
in terms of the $\sup$-convolution
    \begin{equation*}
      (f\star_{\lambda,s} g)(v) = \sup\{
      \mathcal{M}_{\lambda}^s(f(x_1),g(x_2)): v=\lambda x_1+(1-\lambda)x_2 \}
    \end{equation*}
    and the symmetric decreasing rearrangements $f^*$ and $g^*$ of $f$
    and $g$, respectively (defined in \eqref{rearrangement-def}).

\begin{Th}
\label{Extended-Stochastic-PL}
  Let $s\in(-1/n,\infty)$ and let $f,g:\R^n\to[0,\infty)$ be
    integrable $s$-concave functions and $N,M>n+1$. Then for
    $\alpha>0$,
  \begin{equation*}
    \Pro\parenth{\int_{\R^n}([f]_N\star_{\lambda,s}[g]_M)(v)dv
    >\alpha} \geq
    \Pro\parenth{\int_{\R^n}([f^*]_N\star_{\lambda,s}[g^*]_M)(v)dv
    >\alpha}.
  \end{equation*}
\end{Th}
\noindent When $N, M\rightarrow \infty$ one gets
    \begin{equation}
    \label{Isoperimetric-Prekopa-Leindler}
        \displaystyle\int_{\R^n}{(f\star_{\lambda,s} g)(v)}dv
        \geq
        \int_{\R^n}{(f^*\star_{\lambda,s} g^*)(v)}dv.
    \end{equation}
As mentioned, Brascamp and Lieb's first approach to cases of
\eqref{BBL} used rearrangements. Recent interest in rearranged
strengthenings for this and other means in \eqref{BBL} have been
studied by Melbourne \cite{mel} for general functions. Roysdon and
Xing have studied $L_p$ variants of the Borell-Brascamp-Lieb
inequality \cite{RoysXing}.  Our treatment will also allow for other
means (see Remark \ref{rem:C}).  We focus on $s$-concave functions
because in this case there is a stronger local stochastic dominance.

Special cases of Theorem \ref{Extended-Stochastic-PL}, namely
$s=1/q$ ($q\in \mathbb{N})$ were treated in
\cite{PivRB:2019-PreLei}.  The approach used multiple integral
rearrangement inequalities (as discussed above) and, additionally,
built on ideas of Artstein-Avidan, Klartag and Milman
\cite{artklamilsan} on moving from convex sets to $\log$-concave
functions.  Here a new key step is inspired by Rinott's approach to
\eqref{BBL} via epigraphs of convex functions \cite{rinonconvexity};
the latter has recently been used in a dual setting by
Artstein-Avidan, Florentin and Segal \cite{AFS:2019-PolarPL} for a new
Pr\'{e}kopa-Leindler inequality.

Another new tool developed in this paper is that of linear parameter
systems for $s$-concave functions.  Linear parameter systems along a
direction $\theta\in\Sp^{n-1}$ are families of convex sets of the form
\begin{equation*}
  K_t = \conv{x_i+\lambda_i t\theta:i\in I}, \hspace{1em} t\in [a,b],
\end{equation*}
 where $I$ is an index set, $\{x_i\}_{i\in I}\subseteq\R^n$ and
 $\{\lambda_i\}_{i\in I}\subseteq \mathbb{R}$ are bounded sets.
 Rogers and Shephard \cite{rogsheext} proved the fundamental fact that
 the volume of $K_t$ is a convex function of $t$. This was extended to
 the more general notion of shadow systems by Shephard
 \cite{Shephard:1964shadow}. Shadow systems encorporate key features
 of Steiner symmetrization and have been successfully used in a
 variety of isoperimetric type inequalities, developed especially by
 Campi and Gronchi, e.g.,
 \cite{campigronchionvolume,Cam-Gron:2006volume}; for other examples,
 see \cite{fradelizi2012application,saroglou2013shadow,
   meyreishadow,weberndorfer2013shadow} or \cite{Schneider-Book} and
 the references therein.

The proof of Theorem \ref{Extended-Stochastic-PL} relies on extending
linear parameter systems to $s$-concave functions. Let $I$ be an index
set, and $\{(x_i+\lambda_i t\theta,z_i)\}_{i\in I}\subseteq
\mathbb{R}^n\times[0,\infty)$ a collection of points lying under the
  graph of an integrable $s$-concave function and set
  $w_i(t)=x_i+\lambda_i t\theta$, $i\in I$. Analogous to the
  definition of $[f]_N$, we define, for $s\geq 0$, $f_{t,s}$ to be the
  least $s$-concave function above $\{w_i(t)\}$; similarly, for $s<0$,
  we define $f_{t,s}$ to be the greatest $s$-concave function beneath
  the points $\{w_i(t)\}$ (see \S \ref{section:shadow}). In this setting, we show that
    \begin{equation*}
    t\mapsto \int_{\mathbb{R}^n}f_{t,s}(v)dv
  \end{equation*} is convex. 
Just as linear parameter systems can be viewed as special shadow
systems, the same applies to $s$-concave functions. We also give an
interpretation of shadow systems of $s$-concave functions in terms of
associated epigraphs and hypographs and establish related convexity
properties in \S \ref{section:shadow}.  We show in \S
\ref{Steiner-Convexity} how these interface with rearrangement
inequalities and thus provide a path towards stochastic geometry of
$s$-concave functions and associated extremal inequalities.

\section{Preliminaries}
\label{prelim}

We will denote by $\{e_1,\ldots,e_n\}$ the standard basis in
$\R^n$. Let $K$ be a compact, convex set in $\mathbb{R}^n$, $\theta$ on the
unit sphere $\Sp^{n-1}$ and $P:=P_{\theta^{\perp}}$ the orthogonal
projection onto $\theta^{\perp}$. We define $u_K:PK\rightarrow \R$ by
\begin{equation*}
  \label{upperfunction}
u_K(x):=u(K,x):= \sup\{\lambda:x+\lambda\theta\in K\}
\end{equation*}
and $\ell_K:PK\rightarrow \R$ by
\begin{equation*}
  \label{lowerfunction}
\ell_K(x):=\ell(K,x):=\inf\{\lambda:x+\lambda\theta\in K\}.
\end{equation*}
Notice that $u_K$ and $\ell_K$ are concave and
convex, respectively.

We recall that the Steiner symmetral of a non-empty compact set
$A\subseteq\R^n$ with respect to $\theta^{\perp}$,
$S_{\theta^{\perp}}A$, is the set with the property that for each line
$l$ orthogonal to $\theta^{\perp}$ and meeting $A$, the set
$l\cap S_{\theta^{\perp}}A$ is a closed segment with midpoint on
$\theta^{\perp}$ and length equal to that of the set $l\cap A$. The
mapping $S_{\theta^{\perp}}:A\rightarrow S_{\theta^{\perp}}A$ is
called the Steiner symmetrization of $A$ with respect to
$\theta^{\perp}$. In particular, if $K$ is a convex body
\begin{equation*}
	S_{\theta^{\perp}}K = \{x+\lambda \theta:x\in PK,
        -\dfrac{u_K(x)-\ell_K(x)}{2} \leq \lambda \leq
        \dfrac{u_K(x)-\ell_K(x)}{2}\}.
\end{equation*}
This shows that $S_{\theta^{\perp}}K$ is convex, since the function
$u_K-\ell_K$ is concave. Moreover, $S_{\theta^{\perp}}K$ is symmetric with
respect to $\theta^{\perp}$, it is closed, and by Fubini's theorem it
has the same volume as $K$.

Let $A\subseteq\R^n$ be a Borel set with finite Lebesgue measure. The
symmetric rearrangement, $A^*$, of $A$ is the open ball with center at
the origin whose volume is equal to the measure of $A$. Since we
choose $A^*$ to be open, $\ind_{A^*}$ is lower semicontinuous. The
symmetric decreasing rearrangement of $\ind_A$ is defined by
$\ind^*_A=\ind_{A^*}$. We say a Borel measurable function
$f:\R^n\to[0,\infty)$ vanishes at infinity if for every $t>0$, the set
  $\{x\in\R^n:f(x)>t\}$ has finite Lebesgue measure. In such a case,
  the symmetric decreasing rearrangement $f^*$ is defined by
  \begin{equation}
    \label{rearrangement-def}
    f^*(x) = \int_0^{\infty}{\ind^*_{\{f>t\}}(x)}dt =
    \int_0^{\infty}{ \ind_{\{f>t\}^*}(x)}dt.
  \end{equation}
Observe that $f^*$ is radially symmetric, radially decreasing, and
  equimeasurable with $f$, i.e., $\{f>t\}$ and $\{f^*>t\}$ have the
  same measure for each $t>0$. Let $\{e_1,\ldots,e_n\}$ be an
  orthonormal basis of $\R^n$ such that $e_1=\theta$. Then, for $f$
  vanishing at infinity, the Steiner symmetral $f(\cdot|\theta)$ of
  $f$ with respect to $\theta^{\perp}$ is defined as follows: set
  $f_{(x_2,\ldots,x_n),\theta}(t)=f(t,x_2,\ldots,x_n)$ and define
  $f^*(t,x_2,\ldots,x_n|\theta):=(f_{(x_2,\ldots,x_n),\theta})^*(t)$. In
  other words, we obtain $f^*(\cdot|\theta)$ by rearranging $f$ along
  every line parallel to $\theta$. We refer to the books
  \cite{lielosanalysis,simconvexity} or the introductory notes
  \cite{burashort} for further background material on rearrangement of
  functions.

%
%
%
%
%
%
%
%
%
%

\section{Rinott's approach to $s$-concave functions via epigraphs and hypographs}
\label{Epi-con}

Rinott \cite{rinonconvexity} provides a geometric proof of the
Borell-Brascamp-Lieb inequalities \eqref{BBL} by deriving integral
inequalities for functions using certain higher-dimensional
measures. We start this section by recalling his approach.

A function $f:\R^n\to[0,\infty)$ is called $\log$-concave if $\log f$
  is concave on its support. We also use the notion of {\it
    $s$-concavity} as in \cite{Bor:1975-Convex} meaning that $f$ is
  $s$-concave if $f^{s}$ is concave on its support; this differs from
  other uses of this term
  \cite{artklamilsan,klamarg,PivRB:2019-PreLei}. Any $s$-concave
  function, for $s>0$, is also log-concave. For $A\subseteq\R^n$, we
  define the epigraph of $f$ on $A$ in $\R^n$ by
    \begin{equation*}
	   E_A(f) = \{ (x,z)\in A\times\R: f(x)\leq z \}.
    \end{equation*}
Analogously we define the hypograph of $f$ on $A$ by
    \begin{equation*}
	   H_A(f) = \{ (x,z)\in A\times[0,\infty): f(x)\geq z \}.
    \end{equation*}
    When we omit the subscript $A$, we assume that $A$ is the support
    of $f$.

Let $f:\R^n\to[0,\infty)$ be an $s$-concave function for
  $s\in(-1/n,\infty)$, and $\nu$ a measure on $\R^n\times\R$ such that
    \begin{equation}
    \label{Nu-Measure}
        d\nu(x_1,\ldots,x_{n+1}) = h(x_{n+1})\,dx_1\cdots dx_{n+1},
    \end{equation}
for some continuous function $h:\R\to\R$. With this setup, we
  can express the integral of $f$ in terms of the $\nu$-measure of the
  epigraph or hypograph of a transformation of it. Moreover,
    \begin{equation}
    \label{Integral-Measure}
        \int_A{f(x)}dx
        =
        \begin{cases}
            \nu(E_A(f^s)), & \text{for } h(x_{n+1})=-\frac{1}{s}\,x_{n+1}^{\frac{1}{s}-1}, \text{if $s<0$}. \\[0.5em]
            \nu(E_A(-\log{f})), & \text{for } h(x_{n+1})=e^{-x_{n+1}}, \hspace{.9em}\text{if $s=0$}. \\[0.5em]
            \nu(H_A(f^s)), & \text{for } h(x_{n+1})=\frac{1}{s}\,x_{n+1}^{\frac{1}{s}-1}, \hspace{.9em}\text{if $s>0$}.
        \end{cases}
    \end{equation}
Notice the $s$-concavity of $f$ implies the convexity of $H_A(f^s)$, $E_A(-\log{f})$, and $E_A(f^s)$ respectively.

For $x_1,x_2\in \mathbb{R}^n$ and $\lambda\in[0,1]$, let
$x_1+_{\lambda}x_2=\lambda x_1+(1-\lambda)x_2$; similarly for $K,L\subseteq
\mathbb{R}^n$, we write
\begin{equation}
        \label{eqn:plus_lambda}
            K+_{\lambda} L := \lambda K + (1-\lambda)L.
\end{equation}
For convex functions $\varphi,\psi:\R^n\to\R$ and $\lambda\in [0,1]$,
we define their infimal convolution by
\begin{equation*}
(\varphi\square_{\lambda}\psi)(v) =
\underset{v={x_1}+_{\lambda}x_2}{\inf}\{\mathcal{M}_{\lambda}^1\parenth{\varphi(x_1),\psi(x_2)}\},
\end{equation*}
so that
\begin{equation*}
  E(\varphi\square_{\lambda}\psi) = E(\varphi)+_{\lambda}E(\psi).
\end{equation*}
Let $f,g:\R^n\to[0,\infty)$ be $s$-concave functions.  When $s<0$,
  $f^s$ and $g^s$ are convex and
\begin{eqnarray*}
    	(f\star_{\lambda,s}g)^s(v) &=&
  \left(\underset{v={x_1}+_{\lambda}x_2}{\sup}
  \mathcal{M}_{\lambda}^1(f^s(x_1),g^s(x_2))^{1/s}\right)^s \\ &=&
  \underset{v={x_1}+_{\lambda}x_2}{\inf}\mathcal{M}_{\lambda}^1\parenth{{f}^s(x_1),g^s(x_2)}\\ &
  =&(f^s\square_{\lambda}g^s)(v),
\end{eqnarray*}and we have
\begin{equation}
    \label{Convexity-s-Epigraph}
        E((f\star_{\lambda,s}g)^s)
        =
        E(f^s)+_{\lambda}E(g^s).
\end{equation}
For $s=0$, $-\log{f}$ and $-\log{g}$ are convex and
\begin{eqnarray*}
	-\log{(f\star_{\lambda,0}g)}(v) &=&
        -\log\parenth{\underset{v={x_1}+_{\lambda}x_2}{\sup}\{\mathcal{M}_{\lambda}^0(f(x_1),g(x_2))\}}\\ &=
        &\underset{v={x_1}+_{\lambda}x_2}{\inf}\{-\log
        \mathcal{M}_{\lambda}^0(f(x_1),g(x_2))\}\\ &=&
        ((-\log{f})\square_{\lambda}(-\log{g}))(v),
\end{eqnarray*}which implies
\begin{equation}
    \label{Convexity-Log-Epigraph}
        E(-\log{(f\star_{\lambda,0}g)})
        =
        E(-\log{f})+_{\lambda}E(-\log{g}).
\end{equation}
Lastly, for $s>0$, $-f^s$, $-g^s$ are convex and
\begin{eqnarray*}
        -(f\star_{\lambda,s}g)^s(v)
        &= &
        -\underset{v={x_1}+_{\lambda}x_2}{\sup}\mathcal{M}_{\lambda}^1\parenth{{f^s}(x_1),g^s(x_2)}\\
		&=& \underset{v={x_1}+_{\lambda}x_2}{\inf}\mathcal{M}_{\lambda}^1\parenth{{-f^s}(x_1),-g^s(x_2)}\\
        &=&\nonumber
        (-f^s\square_{\lambda}-g^s)(v), 
\end{eqnarray*}from which it follows that
\begin{equation}
    \label{Convexity-Hypograph}
        H((f\star_{\lambda,s}g)^s) = H(f^s)+_{\lambda}H(g^s).
\end{equation}



%
%

\section{Convex hull and $\mathcal{M}$-addition operations}

Let $C\subseteq \mathbb{R}^N$ be a compact convex set; for
$x_1,\ldots,x_N\in\R^n$, we view the $n\times N$ matrix $[x_1,\ldots,x_N]$ as
an operator from $\mathbb{R}^N$ to $\mathbb{R}^n$. Then
\begin{equation}
\label{Conv-Op-Points}
  [x_1,\ldots,x_N]C =\Bigl\{\sum\nolimits_i
    c_ix_i:c=(c_i)\in C\Bigr\}
\end{equation}
produces a convex set in $\R^n$. This viewpoint was used by the
first-named author and Paouris in \cite{paopivprob} in randomized
isoperimetric inequalities for convex sets; for the special case
$C=\conv{e_1,\ldots,e_N}$, one has
    \[
        [x_1,\ldots,x_N]C = \conv{x_1,\ldots,x_N}.
    \]
Moreover, for vectors $x_1\ldots,x_N, x_{N+1}$,$\ldots$, $x_{N+M}$, we
have
\begin{eqnarray}
  \lefteqn{[x_1,\ldots,x_N]C_N+[x_{N+1},\ldots,X_{N+M}]C_M }\nonumber \\
  & &=
  [x_1,\ldots,x_N]C_N + [x_{N+1},\ldots,x_{N+M}]C_M
  \nonumber\\
  & &=
  [x_1,\ldots,x_{N+M}](C_N+\widehat{C}_M),
  \label{eqn:hat_sum}
\end{eqnarray}
where $C_k=\mathop{\rm conv}\{e_1,\ldots,e_k\}$ for $k=N,M$ and
$\widehat{C}_M=\mathop{\rm conv}\{e_{N+1},\ldots,e_{N+M}\}$.  The
convex operations on points \eqref{Conv-Op-Points} can also be
generalized to convex operations on sets by using the notion of
$\mathcal{M}$-addition. This was introduced by Gardner, Hug, and Weil
\cite{GarHugWei:2013-Operations,GarHugWei:2014-Orlicz} as a unifying
framework for operations in Lutwak, Yang, and Zhang's $L_p$ and Orlicz
Brunn-Minkowski theory (see e.g., \cite{LutYanZha:2010-Orlicz}).  For
$\mathcal{M}\subseteq\R^N$ and subsets $K_1,\ldots, K_N$ in $\R^n$,
their {\it $\mathcal{M}$-combination} is defined by
\begin{eqnarray*}
  \oplus_{\mathcal{M}}(K_1,\ldots,K_N) &=& \Bigl\{
  \displaystyle\sum\limits_{i=1}^N m_ix_i : x_i\in K_i,
  (m_1,\ldots,m_N)\in \mathcal{M} \Bigr\}.
\end{eqnarray*}
Thus, with this notation, for $C = \mathcal{M}$,
\begin{equation*}
  \oplus_{C}(\{x_1\},\ldots,\{x_N\}) = [x_1,\ldots,x_N]C.
\end{equation*}
Additionally, when $K_1,\ldots, K_N$ are convex and $\mathcal{M}$ is
compact, convex and contained in the positive orthant or
origin-symmetric, then $\oplus_{\mathcal{M}}(K_1,\ldots,K_N)$ is
convex \cite[Theorem 6.1]{GarHugWei:2013-Operations}.

To connect with the epigraphs and hypographs defined in \S
\ref{Epi-con}, we use $\mathcal{M}$-combinations of rays and line segments
in $\R^{n+1}$. Let $C\subseteq\R^N$ be a compact, convex set contained
in the positive orthant, $\rho_1,\dots,\rho_N\in\R$, and
$x_1,\dots,x_N\in\R^n$. We define the rays
    \begin{subequations}
    \begin{align}
    \label{Ray-Epi}
        R_{\rho_i}(x_i)
        &=
        \{(x_i,r)\in\R^n\times\R: \rho_i\leq r\}\\
\intertext{and the line segments}
    \label{Ray-Hypo}
        \widetilde{R}_{\rho_i}(x_i)
        &=
        \{(x_i,r)\in\R^n\times[0,\infty): \rho_i\geq r\}.
    \end{align}
    \end{subequations}
Accordingly,
    \begin{equation*}
        \oplus_C(R_{\rho_1}(x_1),\ldots,R_{\rho_N}(x_N))
        \hspace{1em} \text{ and } \hspace{1em}
        \oplus_C(\widetilde{R}_{\rho_1}(x_1),\ldots,\widetilde{R}_{\rho_N}(x_N))
    \end{equation*}
form the epigraph of a convex function and the hypograph of a concave
function, respectively. By choosing $C=\conv{e_1,\ldots,e_N}$, one
simply takes the convex hull of the rays or line segments,
respectively.

%
%
%

\section{Shadow systems of $s$-concave functions}

\label{section:shadow}

In this section, we recall the notion of linear parameter and shadow
systems of convex sets and extend these to $s$-concave functions.  We
establish a corresponding convexity property in the functional
setting.

Recall the notation for linear parameter systems from the
introduction: for an index set $I$, bounded sets $\{x_i\}_{i\in
  I}\subseteq\R^n$ and $\{\lambda_i\}_{i\in I}\subseteq \mathbb{R}$
and $\theta\in\Sp^{n-1}$, $a\leq t\leq b$, we set
    \[
        K_t = \conv{x_i+\lambda_i t\theta:i\in I}, \hspace{1em} t\in
        [a,b].
    \]
In the notation of \S \ref{prelim}, we use $u(K_t,\cdot)$,
$\ell(K_t,\cdot)$, $P=P_{\theta^{\perp}}$, and we set $D=PK$. Define
$L(x)=u(K_t,x)-\ell(K_t,x)$. Rogers and Shephard proved the
fundamental fact that for $x\in D$, $t\mapsto L(K_t,x)$ is
convex. Consequently, the map
\begin{equation}
  \label{eqn:RS_convex}
t\mapsto \lvert K_t \rvert =\int_{D}L(K_t,x)dx
\end{equation}
is a convex function of $t$. For background on shadow systems, see
e.g., \cite{campigronchionvolume, meyreishadow, saroglou2013shadow},
and \cite[\S 10.4]{Schneider-Book}.

We will use certain linear parameter systems for a finite index set
$I=\{1,\ldots,N\}$ associated to the $\oplus_C$ operations of the
previous section for epigraphs and hypographs.

\begin{Prop}
  \label{prop:lps_EH}
  Let $C$ be a compact convex set in $\mathbb{R}^N$ contained in the
  positive orthant. Let $\{x_i\}_{i=1}^N\subseteq \mathbb{R}^n$,
  $\{z_i\}\subseteq (0,\infty)$ and $\{\lambda_i\}_{i=1}^N\subseteq
    \mathbb{R}$.  For $s\not=0$, let $\rho_i(s) = z_i^s$ and for $s=0$
    let $\rho_i(s)=-\log z_i$. For $a<b$ and $t\in [a,b]$, let
\begin{eqnarray*}
    \label{Shadow-Epigraphs}
        E^s_t =
        \oplus_C(\{{R}_{\rho_i(s)}(x_i+\lambda_it\theta)\}_{i=1}^N)
        \quad (s\geq 0)
\end{eqnarray*}
and 
\begin{equation*}
    \label{Shadow-Hypograph}
        H^s_t = \oplus_{C}(\{\widetilde{R}_{\rho_i(s)}(x_i+\lambda_i
        t\theta)\}_{i=1}^N) \quad (s<0).
\end{equation*}
Then for $\nu$ as in \eqref{Integral-Measure}, $t \mapsto \nu(E^s_t)$
and $t\mapsto \nu(H^s_t)$ are convex.
\end{Prop}

\begin{proof}Let $s\geq 0$. For $z\in \mathbb{R}$, write
$\pi_z=e_{n+1}^{\perp}+ze_{n+1}$ so that
\begin{equation*}
\nu(E^s_t) = \int_{\mathbb{R}} \lvert E^s_t\cap \pi_z\rvert h(z) dz.
\end{equation*}
Thus it suffices to show that for fixed $z$, $t\mapsto \lvert
E^s_t\cap \pi_z\rvert$ is convex.  We have
\begin{eqnarray*}
  E^s_t&=& \left\{\sum_{i=1}^N
  c_i(x_i+r_ie_{n+1})+\left(\sum_{i=1}^N c_i \lambda_i\right)t\theta:
  c\in C, \rho_i(s)\leq r_i\right\}.
\end{eqnarray*}
As noted in the previous section, since $C$ is compact, convex and
contained in the positive orthant, $E_t^s$ is convex.  For $c\in C$,
we write $x_c =\sum_{i=1}^N c_ix_i$, $r_c=\sum_{i=1}^N c_ir_i$, and
$\lambda_c=\sum_{i=1}^N c_i\lambda_i$. For $x_c+r_c e_{n+1} +\lambda_c
t\theta\in E_t^s\cap \pi_z$, we have $r_c=z$ and the
sets $\{x_c+ze_{n+1}\}$, $\{\lambda_c\}$ are bounded. Thus $E_t^s\cap
\pi_z$ is a linear parameter system of convex sets indexed by $C$.
Then $t\mapsto \lvert E^s_{t} \cap \pi_z\rvert $ is convex by
\eqref{eqn:RS_convex}. The argument for $s<0$ is analogous.
\end{proof}

Next, we define a linear parameter system of $s$-concave
functions. Let $I$ be an index set, and $\{(x_i,z_i)\}_{i\in
  I}\subseteq \mathbb{R}^n\times[0,\infty)$ a collection of points
  under the graph of some integrable $s$-concave function. For $s\geq
  0$, let $T_{\{w_i\},s}$ and be the least $s$-concave function above
  $\{w_i\}$.  For $s<0$, let $T_{\{w_i\},s}$ be the greatest
  $s$-concave function beneath $\{w_i\}.$ More explicitly,
  $T_{\{w_i\},s}$ is supported on $\mathop{\rm conv}\{x_i:i\in I\}$
  and given by
    \begin{subequations}
    \begin{align*}
      T_{\{w_i\},s}(x) &=
        \begin{cases}
            \inf\{z^{1/s}:(x,z)\in \mathcal{P}_{\{w_i\}}\}, & \text{if $s<0$}
            \\[.2em] \sup\{e^{z}:(x,z)\in \mathcal{P}_{\{w_i\}}\}, & \text{if
              $s=0$}, \\[.2em] \sup\{z^{1/s}:(x,z)\in \mathcal{P}_{\{w_i\}}\}, &
            \text{if $s>0$},
        \end{cases}
\intertext{where}
        \mathcal{P}_{\{w_i\}}
        &=
        \begin{cases}
            {\rm conv}\{(x_i,{z^s_i})\}_{i\in I}, & \text{if
              $s\neq0$}. \\[.2em] {\rm conv}\{(x_i,\log{z_i})\}_{i\in
              I}, & \text{if $s=0$}. 
        \end{cases}
    \end{align*}
\end{subequations}
With the above notation, we assume that $w_i(t) = (x_i+\lambda_i
t\theta,z_i)$, $a\leq t\leq b$. Then setting
$f_{t,s}=T_{\{w_i(t)\},s}$, we call the family $\{f_{t,s}\}$ a linear
parameter system of $s$-concave functions. The convexity
property corresponding to \eqref{eqn:RS_convex} reads as follows.

\begin{Prop} 
  \label{prop:lpsconcave}
  Let $\{f_{t,s}\}$ be a linear parameter system of $s$-concave
  functions. Then
  $$ t\mapsto \int_{\mathbb{R}^n} f_t(v)dv
  $$ is a convex function.
\end{Prop}

\begin{proof} 
As for linear parameter systems of convex sets, we can assume without
loss generality that $I$ is finite, say $I=\{1,\ldots,N\}$. We take
$C$ to be $\mathop{\rm conv}\{e_1,\ldots,e_N\}$. In the notation of
the previous proposition, we have
\begin{equation*}
  E_t^s = \mathop{\rm
    conv}\{{R}_{\rho_i(s)}(x_i+\lambda_it\theta)\}_{i=1}^N
\end{equation*} and $E_t^s = E(-\log f_{t,s})$ for $s=0$, while 
$E_t^s = E(f_{t,s}^s)$ for $s<0$.  Similarly,
\begin{equation*}
    H^s_t = \mathop{\rm
      conv}\{\widetilde{R}_{\rho_i(s)}(x_i+\lambda_it\theta)\}_{i=1}^N,
\end{equation*}hence $H^s_t = H(f_{t,s}^s)$ for $s>0$.
By \eqref{Integral-Measure}, we have
\begin{equation*}
  \int f_{t,s}(v)dv =
        \begin{cases}
            \nu(E^s_t), & \text{for }
            h(x_{n+1})=-\frac{1}{s}\,x_{n+1}^{\frac{1}{s}-1}, \text{if
              $s<0$}. \\[0.5em] \nu(E^s_t), & \text{for }
            h(x_{n+1})=e^{-x_{n+1}}, \hspace{.9em}\text{if
              $s=0$}. \\[0.5em] \nu(H^s_t), & \text{for }
            h(x_{n+1})=\frac{1}{s}\,x_{n+1}^{\frac{1}{s}-1}, \hspace{.9em}\text{if
              $s>0$},
        \end{cases}
\end{equation*}
and we can conclude the proof by applying Proposition
\ref{prop:lps_EH}.
\end{proof}

Shephard \cite{Shephard:1964shadow} introduced shadow systems to
extend linear parameter systems.  Given a convex body
$K\subseteq\R^n$, a bounded function $\alpha:K\to\R$ and
$t\in[a,b]$, a shadow system in direction
$\theta\in\Sp^{n-1}$ is a family of convex sets
    \begin{equation}
        K_t
        =
        \conv{x+\alpha(x)t\theta : x\in K}.
    \end{equation}
Then $K_t=P_t\overline{K}$, where $\overline{K}=\{(x,\alpha(x)):x\in
K\}$ and $P_t:\mathbb{R}^{n}\times\mathbb{R}\rightarrow \mathbb{R}^n$
is the projection parallel to $e_{n+1}-t\theta$ given by $P_t(x,y) =
x+ty\theta$. Conversely, for any convex body
$\overline{K}\subseteq\R^{n+1}$, $\theta\in\Sp^{n-1}$, and
$t\in[a,b]$, the family $\{K_t\}=\{{P}_t\overline{K}\}\subseteq\R^n$
is a shadow system of convex bodies.

The correspondence between linear parameter systems of $s$-concave
functions and epigraphs/hypographs in the proof of Proposition
\ref{prop:lpsconcave} affords a similar extension to shadow systems.
We can simply start with shadow systems of epigraphs or hypographs.
Let $\mathcal{E}\subseteq\R^{n+2}$ be the epigraph of a convex function
$\Phi:\R^{n+1}\to[0,\infty)$, $\theta\in\Sp^{n-1}$, $t\in[a,b]$, and
  the projection $P_t$ from $\R^n\times\R\times\R$ onto $\R^n\times\R$
  parallel to $e_{n+2}-t\theta$ given by
    \begin{equation}
    \label{Epi-Projection}
        P_t(x,z,y)
        =
        (x+ty\theta,z).
    \end{equation} 
Then the family $\{E_t\}=\{P_t\mathcal{E}\}$ is a shadow system of
epigraphs of convex functions, where
     \begin{equation}
     \label{Shadow-System-Epigraphs}
        E_t = \mathop{\rm conv}\{
        (x+yt\theta,z)\in\R^n\times\R : (x,y,z)\in \mathcal{E} \}.
     \end{equation}
     Consequently, let $\phi:\R^n\to[0,\infty)$ be a convex function,
       $\alpha:E(\phi)\to\R$ a function such that $\alpha|_{E(\phi)\cap
         \pi_z}:\R^n\to\R$ is bounded for all $z\in\R$. Consider
       $\Phi:\R^{n+1}\to\R$ described by its epigraph
       $E(\Phi)=\mathop{\rm conv}\{(x,z,\alpha(x,z)):\phi(x)\leq z\}$
       and set
       \begin{align*}
		(E(\phi))_t
		=
		P_tE(\Phi)
		=
		\mathop{\rm conv}\{(x+\alpha(x,z)t\theta,z):\phi(x)\leq z\}.
	\end{align*}
	Then we define the shadow system of convex functions
        $\phi_t:\R^n\to[0,\infty)$ in direction $\theta\in\Sp^{n}$ by
    \begin{equation}
    \label{Shadow-System-Functions}
        \phi_t(x)
        =
        \inf\{z : (x,z)\in (E(\phi))_t\},
    \end{equation}
Analogously, given the hypograph $\mathcal{H}\subseteq\R^{n+2}$ of a
concave function $\Psi:\R^{n+1}\to[0,\infty)$ the family
  $\{H_t\}=\{P_t\mathcal{H}\}$ is a shadow system of hypographs of
  concave functions where
     \begin{equation}
     \label{Shadow-System-Hypographs}
        H_t = \mathop{\rm conv}\{ (x+yt\theta,z)\in\R^n\times\R :
        (x,z,y)\in \mathcal{H} \}.
     \end{equation}
Similarly, for a concave function $\psi:\R^n\to[0,\infty)$,
  $\alpha:H(\psi)\to\R$ a function such that $\alpha|_{H(\psi)\cap
    \pi_z}:\R^n\to\R$ is bounded for all $z\in\R$. Consider the
  function $\Psi:\R^{n+1}\to\R$ with hypograph
  $H(\Psi)=\{(x,z,\alpha(x,z)):0\leq z \leq \psi(x)\}$ and set
\begin{align*}
(H(\psi))_t
=
P_tH(\Psi)
=
\mathop{\rm conv}\{(x+\alpha(x,z)t\theta,z):0\leq z \leq\psi(x)\}.
\end{align*}
Then we define the shadow system of concave functions
$\psi_t:\R^n\to\R$ by
    \begin{equation}
    \label{Shadow-System-Concave-Functions}
        \psi_t(x) = \sup\{z : (x,z)\in (H(\psi))_t\}.
    \end{equation}

\begin{Prop}
\label{Parameter-Convexity}
For shadow systems of epigraphs $E_t$ and hypographs $H_t$ and $\nu$
as in \eqref{Integral-Measure}, we have that $t \mapsto \nu(E_t)$ and
$t\mapsto \nu(H_t)$ are convex.
\end{Prop}

\begin{proof}
  Let $\nu$ be as in \eqref{Integral-Measure}. Then
  \[
  \nu(E_{t}) = \int_{\R} {|E_{t} \cap \pi_z| h(z)} dz.
  \]
  For each $z$, the restriction of the epigraph $E_t$ to the parallel
  hyperplane $\pi_z$ is a shadow systems of convex bodies. Thus the
  convexity of $t\mapsto \nu(E_{t})$ follows from the convexity of the
  function \eqref{eqn:RS_convex}. The proof for $H_t$ is analogous.
\end{proof}

%
%
%

\section{Random epigraphs and hypographs}
\label{Ran-Fun-Epi}

In this section, we take our stochastic model for $[f]_N$, as defined
in the introduction, and reformulate it in terms of epigraphs and
hypographs.  Thus for an integrable $s$-concave function
$f:\R^n\to[0,\infty)$, we sample independent random vectors
  $(X_1,Z_1),\ldots,(X_N,Z_N)$ according to Lebesgue measure on
    \begin{equation}
    \label{eqn:graph}
        G_f:=\{(x,z)\in\R^n\times[0,\infty):x\in{\rm supp}f, z\leq
          f(x)\}.
    \end{equation}
    For $C_N=\mathop{\rm conv}\{e_1,\ldots,e_N\}$, we set
   \begin{subequations}
    \begin{align}
    \label{Random-Epigraphs}
        [E^s]_N
        &=
        \begin{cases}
            \oplus_{C_N}({R}_{Z^s_1}(X_1),\dots,{R}_{Z^s_N}(X_N)),
            \hfill \text{if
              $s<0$},\\ \oplus_{C_N}({R}_{-\log{Z_1}}(X_1),\dots,{R}_{-\log{Z_N}}(X_N)),\hfill
            \text{if $s=0$}.
        \end{cases}\\
    \label{Random-Hypograph}
        [H^s]_N &=
        \oplus_{C_N}(\widetilde{R}_{Z^s_1}(X_1),\dots,\widetilde{R}_{Z^s_N}(X_N)),\hspace{4.3em}\text{if
          $s>0$}.
    \end{align}
    \end{subequations}
With this notation, 
    \begin{alignat*}{2}
        [E^s]_N&=E([f]_N^s), & \text{if $s<0$}.\\[.5em]
        [E^s]_N&=E(-\log{[f]_N}), \hspace{1em} & \text{if $s=0$}.\\[.5em]
        [H^s]_N&=H([f]_N^s), & \text{if $s>0$}
    \end{alignat*}
and, by \eqref{Integral-Measure}, 
    \begin{subequations}
    \begin{align}
    \label{Random-Integral-Measure}
        \int{[f]_N} &=
        \begin{cases}
            \nu([E^s]_N), & \text{for } h(x_{n+1})=-\frac{1}{s}  x_{n+1}^{\frac{1}{s}-1}, \hfill \text{if $s<0$}. \\[0.5em]
            \nu([E^s]_N), & \text{for } h(x_{n+1})=e^{-x_{n+1}},  \hfill \text{if $s=0$}. \\[0.5em]
            \nu([H^s]_N), & \text{for } h(x_{n+1})=\frac{1}{s}  x_{n+1}^{\frac{1}{s}-1}, \hfill \text{if $s>0$}.
        \end{cases}
    \end{align}
    \end{subequations}

%
%
%
%
%
%
%
%
%
%

\section{Multiple integral rearrangement inequalities}

\label{Steiner-Convexity}

\subsection{Rearrangements and Steiner convexity}

In this section, we show that the multiple integral rearrangement
inequalities of Rogers \cite{Rog:1957-Rearrangement}, and Brascamp,
Lieb, and Luttinger \cite{BLL} interface well with our approach. In
particular, Christ's version \cite{christestimates} of the latter
inequalities is especially applicable; as in \cite{paopivrand}, the
following formulation is convenient for our purpose.

\begin{Th}
  \label{RBLLC}
  Let $f_1,\ldots,f_N$ be non-negative integrable functions on
  $\mathbb{R}^n$ and $F:(\mathbb{R}^n)^N\rightarrow [0,\infty)$. Then
   \begin{eqnarray*}
    \lefteqn{\int\limits_{(\R^n)^N} F(x_1,\ldots,x_N)\prod_{i=1}^N
      f_i(x_i)dx_1\ldots dx_N}\\ & \geq & \int\limits_{(\R^n)^N}
    F(x_1,\ldots,x_N)\prod_{i=1}^N f^*_i(x_i)dx_1\ldots dx_N,
  \end{eqnarray*}whenever  $F$ satisfies the following condition:
  for each $\theta\in\Sp^{n-1}$ and all
  $Y:=\{y_1,\ldots,y_N\}\subseteq\theta^{\perp}$, the function
  $F_Y:\R^N\to[0,\infty)$ defined by
    \begin{equation*}
      F_{Y,\theta}(t_1,\ldots,t_N):=F(y_1+t_1\theta,\ldots,y_N+t_N\theta)
    \end{equation*} is even and quasi-convex.
\end{Th}

The condition on $F$, namely {\it Steiner convexity}, allows the
theorem to be proved via iterated Steiner symmetrization; notice this
terminology differs from the one in \cite{christestimates}. Of special
interest, this condition interfaces well with shadow systems, e.g.,
\cite{camcolgroanote,rogsheext}; see \cite{paopivrand} for further
background and references.


\begin{Prop}
  \label{SC-of-rays-with-fixed-heights-epi}
  Let $\rho_1,\ldots,\rho_N\in\R$ and $C$ a compact convex set contained in
  the positive orthant. Then the function $F:(\R^n)^N\rightarrow
  [0,\infty)$ defined by
    \begin{equation}
    \label{Epi-Rays-Functional}
        F(x_1,\ldots,x_N) = \num{ \oplus_C( R_{\rho_1}(x_1),\ldots,
          R_{\rho_N}(x_N) ) }
    \end{equation}
  is Steiner convex.
\end{Prop}

\begin{proof} 
    Let $\theta\in\Sp^{n-1}$ and $y_1,\ldots,y_N\in\theta^{\perp}$.
    Let $(s_1,\ldots,s_N), (t_1,\ldots,t_N)\in \mathbb{R}^N$ and
    $\tau\in (0,1)$.  For $i=1,\ldots,N$, write $$y_i+(\tau s_i +
    (1-\tau)t_i)\theta = (y_i+t_i\theta)+\tau(t_i-s_i)\theta$$ and
    apply Proposition \ref{prop:lps_EH} with $x_i=y_i+t_i\theta$,
    $\lambda_i=t_i-s_i$ and $t=\tau$ to obtain the convexity in
    $\tau$.  Lastly, the sets
        \[
            \oplus_C(R_{\rho_1}(y_1+t_1\theta),\ldots,R_{\rho_N}(y_N+t_N\theta))
        \]
    and
        \[
            \oplus_C(R_{\rho_1}(y_1-t_1\theta),\ldots,R_{\rho_N}(y_N-t_N\theta))
        \]
    are reflections of one another and so the evenness condition
    required for Steiner convexity holds.
\end{proof}

Next, we state the analogous proposition involving the line segments
\eqref{Ray-Hypo}; the proof follows the same line.

\begin{Prop}
  \label{SC-of-rays-with-fixed-height-hypo}
  Let $\rho_1,\ldots,\rho_N\in\R$ and $C$ a compact convex set
  contained in the positive orthant. Then the function
  $F:(\R^n)^N\rightarrow [0,\infty)$ defined by
    \begin{equation}
    \label{Hypo-Rays-Functional}
        F(x_1,\ldots,x_N)
        =
        \num{
        \oplus_C(
        \widetilde{R}_{\rho_1}(x_1),\ldots, \widetilde{R}_{\rho_N}(x_N)
        )
        }
    \end{equation}
  is Steiner convex.
\end{Prop}

%
%

\section{Main proof}

\begin{proof}[Proof of Theorem~\ref{Extended-Stochastic-PL}]
Let $w_i=(x_i,z_i)\in \mathbb{R}^n\times\mathbb{R}$ for
$i=1,\ldots,M+N$.  For $s\geq 0$, let $T^{(N)}_{\{w_i\},s}$ and
$T^{(M)}_{\{w_i\},s}$ be the least $s$-concave functions above the
collections $\{w_i\}_{i\leq N}$ and $\{w_i\}_{N+1\leq i\leq M}$,
respectively; similarly, when $s<0$, let $T^{(N)}_{\{w_i\}}$ and
$T^{(M)}_{\{w_i\}}$ be the greatest $s$-concave functions beneath the
respective collections $\{w_i\}_{i\leq N}$ and $\{w_i\}_{N+1\leq i\leq M}$. With
this notation, we set
\begin{equation*}
  F(w_1,\ldots,w_{N+M})=\int_{\R^n}T^{(N)}_{\{w_i\}}\star_{\lambda}
  T^{(M)}_{\{w_i\}}(v)dv.
\end{equation*}
Let $f,g:\R^n\to[0,\infty)$ be integrable $s$-concave functions for
  $s\in(-1/n,\infty)$. Sample independent random vectors
  $W_i=(X_i,Z_i)$, $i=1,\dots,N+M$ uniformly according to the Lebesgue
  measure on $G_f$ for $i=1,\ldots,N$ and $G_g$ for
  $i=N+1,\ldots,N+M$. Then the random functions $[f]_{N}$, $[g]_{M}$
  satisfy
    \begin{eqnarray}
        \nonumber
        \lefteqn{\Pro\left(
        \int_{\R^n}[f]_N\star_{\lambda,s}[g]_M(v)
        dv>\alpha
        \right)}\\
        \label{Prob-Funct}
        & & = \frac{1}{\prod_{i=1}^{M+N}\norm{k_i}_1}
        \int_{N+M}\mathds{1}_{\{F>\alpha\}}(\overline{w})
        \prod\limits_{i=1}^{N+M}\mathds{1}_{[0,k_i(x_i)]}(z_i) d\overline{w},
    \end{eqnarray}
where $\int_{N+M}$ is the integral on $(\R^n\times[0,\infty))^{N+M}$, $k_i=f$ for $i=1,\ldots, N$, $k_i=g$ for $i=N+1,\ldots,N+M$, and
    \begin{equation}
    \label{eqn:w}
      \overline{w}=(w_1,\ldots,w_{N+M}), \quad d\overline{w} = dw_1\ldots
      dw_{N+M}.
    \end{equation}
Also we write $C_N=\conv{e_1,\dots,e_N}$ and
$\widehat{C}_M=\conv{e_{N+1},\dots,e_{N+M}}$, and consider $\nu$ as in
\eqref{Integral-Measure} for each case.  \\[.5em]
\noindent{\bf Case $\mathbf{s>0}$:}
By \eqref{Random-Hypograph} and \eqref{Convexity-Hypograph} it follows
    \begin{align*}
        H(([f]_{N}\star_{\lambda,s}[g]_{M})^s)
        &=
        \oplus_{C_N}
        (\{\widetilde{R}_{Z^s_i}(X_i)\}_{i=1}^{N})
        +_{\lambda}
        \oplus_{\widehat{C}_M}
        (\{\widetilde{R}_{Z^s_{i}}(X_{i})\}_{i=N+1}^{N+M})\\[.5em]
        &=
        \oplus_{C_N+_{\lambda}\widehat{C}_M}
        (\{\widetilde{R}_{Z^s_{i}}(X_{i})\}_{i=1}^{N+M}),
    \end{align*}
hence
    \begin{equation}
    \label{Eqn-Int-Rays-Tilde}
        \int_{\R^n}[f]_N\star_{\lambda,s}[g]_M(v) dv
        =
        \num{\oplus_{C_N+_{\lambda}\widehat{C}_M}
        (\{\widetilde{R}_{Z^s_{i}}(X_{i})\}_{i=1}^{N+M})}.
    \end{equation}
By \eqref{Prob-Funct}, Fubini, Proposition~\ref{SC-of-rays-with-fixed-heights-epi}, and Theorem \ref{RBLLC}, we have
\begin{eqnarray*}
  \lefteqn{\Pro\left(
        \int_{\R^n}[f]_N\star_{\lambda,s}[g]_M(v)
        dv>\alpha\right)}\\
        & &=
        \frac{1}{\prod_{i=1}^{M+N}\norm{k_i}_1^{N+M}}\displaystyle
  \int_{(\mathbb{R}^n\times[0,\infty))^{N+M}}\mathds{1}_{\{F>\alpha\}}(\overline{w})
    \prod\limits_{i=1}^{N+M}\mathds{1}_{[0,k_i(x_i)]}(z_i)
    d\overline{w}\\
        & & =
    \frac{1}{\prod_{i=1}^{M+N}\norm{k_i}_1^{N+M}}\int_{[0,\infty)^{N+M}}{\left(
        \int_{(\R^n)^{N+M}}{\mathds{1}_{\{F>\alpha\}}(\overline{w})
          \prod\limits_{i=1}^{N+M}{\mathds{1}_{[0,k_i(x_i)]}(z_i)}}d\overline{x}
        \right)}d\overline{z}\\
        & & \geq
      \frac{1}{\prod_{i=1}^{M+N}\norm{k_i^*}_1^{N+M}}\int_{[0,\infty)^{N+M}}{\left(
          \int_{(\R^n)^{N+M}}{\mathds{1}_{\{F>\alpha\}}(\overline{w})
            \prod\limits_{i=1}^{N+M}{\mathds{1}_{[0,k_i^*(x_i)]}(z_i)}}d\overline{x}
          \right)}d\overline{z}\\
        & & =
        \frac{1}{\prod_{i=1}^{M+N}\norm{k_i^*}_1^{N+M}}\displaystyle
        \int_{(\mathbb{R}^n\times[0,\infty))^N}\mathds{1}_{\{F>\alpha\}}(\overline{w})
          \prod\limits_{i=1}^{N+M}\mathds{1}_{[0,k_i^*(x_i)]}(z_i)
          d\overline{w}\\
        & & =
        \Pro\left(
        \int_{\R^n}[f^*]_N\star_{\lambda,s}[g^*]_M(v)
        dv>\alpha\right).
\end{eqnarray*}	

\hspace{1em}

\noindent{\bf Case $\mathbf{s=0}$:} Using \eqref{Random-Epigraphs}, \eqref{Convexity-Log-Epigraph}, we have
    \begin{equation*}
        E(-\log{([f]_{N}\star_{\lambda,0}[g]_{M})})
        =
        \oplus_{C_N+_{\lambda}\widehat{C}_M}
        (\{{R}_{-\log{Z_i}}(X_{i})\}_{i=1}^{N+M}),
    \end{equation*}
so
    \begin{equation}
    \label{Eqn-Int-Rays-Log}
        \int_{\R^n}[f]_N\star_{\lambda,0}[g]_M(v) dv
        =
        \num{\oplus_{C_N+_{\lambda}\widehat{C}_M}
        (\{{R}_{-\log{Z_i}}(X_{i})\}_{i=1}^{N+M})}.
    \end{equation}
It follows as before
\begin{equation*}
    \Pro
    \left(
    \int_{\R^n}[f]_N\star_{\lambda,0}[g]_M(v)dv>\alpha
    \right)
    \geq
    \Pro
    \left(
    \int_{\R^n}[f^*]_N\star_{\lambda,0}[g^*]_M(v)dv>\alpha
    \right).
\end{equation*}	

\hspace{1em}

\noindent{\bf Case $\mathbf{s<0}$:} It follows by \eqref{Random-Epigraphs} and \eqref{Convexity-Log-Epigraph} that
    \begin{equation*}
        E(([f]_{N}\star_{\lambda,s}[g]_{M})^s)
        =
        \oplus_{C_N+_{\lambda}\widehat{C}_M}
        (\{R_{Z^s_i}(X_{i})\}_{i=1}^{N+M}),
    \end{equation*}
hence
    \begin{equation}
    \label{Eqn-Int-Rays-s}
        \int_{\R^n}[f]_N\star_{\lambda,s}[g]_M(v) dv
        =
        \num{\oplus_{C_N+_{\lambda}\widehat{C}_M}
        (\{R_{Z^s_i}(X_{i})\}_{i=1}^{N+M})}.
    \end{equation}
It follows as before
\begin{equation*}
    \Pro
    \left(
    \int_{\R^n}[f]_N\star_{\lambda,s}[g]_M(v)dv>\alpha
    \right)
    \geq
    \Pro
    \left(
    \int_{\R^n}[f^*]_N\star_{\lambda,s}[g^*]_M(v)dv>\alpha
    \right).
\end{equation*}	
\end{proof}

\begin{Rem}
  \label{rem:C}
We have applied Propositions \ref{SC-of-rays-with-fixed-heights-epi}
and \ref{SC-of-rays-with-fixed-height-hypo} only in the special case
when $C=C_N +_{\lambda}\widehat{C}_M$. Since these propositions apply
to more general convex sets $C$, they can be used to treat alternate
means and different stochastic functions in Theorem
\ref{Extended-Stochastic-PL}.  This direction and its geometric
implications are outside of our present scope but will appear in a
forthcoming work of the authors.
\end{Rem}
%
%

\addcontentsline{toc}{section}{References} \bibliographystyle{plain}
\bibliography{biblio}

\end{document}